\documentclass[a4paper,oneside,10pt]{amsart}

\usepackage{a4wide}
\usepackage[T1]{fontenc}
\usepackage[ansinew]{inputenc}
\usepackage{lmodern} 
\usepackage{graphicx}
\usepackage{amsmath}
\allowdisplaybreaks
\usepackage{amsthm}
\usepackage{amsfonts}
\usepackage{amssymb}
\usepackage{setspace}
\usepackage{mathrsfs}
\usepackage[all]{xy}
\usepackage{enumerate}
\usepackage{xcolor}
\usepackage{catchfile}
\usepackage{autobreak}
\usepackage[pagewise]{lineno}
\DeclareMathSymbol{\shortminus}{\mathbin}{AMSa}{"39}
\usepackage{tikz}

\theoremstyle{plain}

\newtheorem{theorem}{Theorem}[section]

\newtheorem{proposition}[theorem]{Proposition}
 \newtheorem{lemma}[theorem]{Lemma}

\theoremstyle{definition}
\newtheorem{remark}[theorem]{Remark}

\newtheorem{assumptions}[theorem]{Assumption}

 \newtheorem{example}[theorem]{Example}
 \newtheorem{definition}[theorem]{Definition}

\newtheorem*{proposition*}{Proposition}
\newtheorem*{definition*}{Definition}
\numberwithin{equation}{section}
\theoremstyle{plain}
\newtheorem*{theorem*}{Theorem}

\newenvironment{abc}{\begin{enumerate}[{\rm (a)}]}{\end{enumerate}}
\newenvironment{num}{\begin{enumerate}[{\rm 1.}]}{\end{enumerate}}
\newenvironment{iiv}{\begin{enumerate}[{\rm (i)}]}{\end{enumerate}}

\def\dom{\mathrm{D}}

\def\dd{\mathrm{d}}

\def\ee{\mathrm{e}}
\def\vv{\mathrm{v}}
\def\RR{\mathbb{R}}
\def\CC{\mathbb{C}}

\def\NN{\mathbb{N}} 
\def\LLL{\mathscr{L}}

\def\L{\mathrm{L}}

\def\BC{\mathrm{C}_{\mathrm{b}}}
\def\BC{\mathrm{C}_{\mathrm{b}}}

\def\Ell{\mathrm{L}}

\def\esssup{\mathop{\mathrm{ess}\,\sup}}
\def\diag{\mathrm{diag}}
\def\Cnull{\mathrm{C}_0}



\makeatletter
\newcommand{\opnorm}{\@ifstar\@opnorms\@opnorm}
\newcommand{\@opnorms}[1]{%
  \left|\mkern-1.5mu\left|\mkern-1.5mu\left|
   #1
  \right|\mkern-1.5mu\right|\mkern-1.5mu\right|
}
\newcommand{\@opnorm}[2][]{%
  \mathopen{#1|\mkern-1.5mu#1|\mkern-1.5mu#1|}
  #2
  \mathclose{#1|\mkern-1.5mu#1|\mkern-1.5mu#1|}
}
\makeatother

\begin{document}
\title{A Lumer--Phillips type generation theorem for bi-continuous semigroups}

\author{Christian Budde\hspace{0.5pt}\MakeLowercase{$^{\text{1}}$} and Sven-Ake Wegner\hspace{0.5pt}\MakeLowercase{$^{\text{2}}$}}

\renewcommand{\thefootnote}{}
\hspace{-1000pt}\footnote{\hspace{5.5pt}2020 \emph{Mathematics Subject Classification}: Primary 47B44; Secondary: 47D03, 47D06, 34G10, 46A70, 46A03.}

\hspace{-1000pt}\footnote{\hspace{5.5pt}\emph{Key words and phrases}: Bi-continuous semigroups, Lumer--Phillips Theorem, dissipativity, evolution equation.\vspace{1.6pt}}

\hspace{-1000pt}\footnote{\hspace{0pt}$^{1}$\,University of the Free State, Department of Mathematics and Applied Mathematics, Faculty of Natural and\linebreak\phantom{x}\hspace{1.2pt}Agriculture Sciences, PO Box 339, Bloemfontein 9300, South Africa, E-Mail: buddecj@ufs.ac.za\vspace{1.6pt}}

\hspace{-1000pt}\footnote{\hspace{0pt}$^{2}$\,Corresponding author: Universit\"at Hamburg, Department of Mathematics, Bundesstra\ss{}e 55, 20146 Hamburg,\linebreak\phantom{x}\hspace{1.2pt}Germany, Phone:\hspace{1.2pt}\hspace{1.2pt}+49\hspace{1.2pt}40\hspace{1.2pt}/\hspace{1.2pt}42838\hspace{1.2pt}-\hspace{1.2pt}7657, Fax:\hspace{1.2pt}\hspace{1.2pt}+49\hspace{1.2pt}40\hspace{1.2pt}/\hspace{1.2pt}42838\hspace{1.2pt}-\hspace{1.2pt}5117 E-Mail: sven.wegner@uni-hamburg.de.\vspace{1.6pt}}

\begin{abstract}                  
The famous 1960s Lumer--Phillips Theorem states that a closed and densely defined operator $A\colon \dom(A)\subseteq X\rightarrow X$ on a Banach space $X$ generates a strongly continuous contraction semigroup if and only if $(A,\dom(A))$ is dissipative and the range of $\lambda-A$ is surjective in $X$ for some $\lambda>0$. In this paper, we establish a version of this result for bi-continuous semigroups and apply the latter amongst other examples to the transport equation as well as to flows on infinite networks.
\end{abstract}

\date{}
\maketitle

\section{Introduction}

In this paper we consider abstract Cauchy problems of the form
\begin{align}\label{eqn:ACP}\tag{ACP}
\begin{cases}
\dot{x}(t)=Ax(t) \text{ for } t\geq0,\\
x(0)=x_0&
\end{cases}
\end{align}
where $A\colon \dom(A)\subseteq X\rightarrow X$ is an unbounded linear operator on a Banach space $X$ and $x_0\in X$ is some initial condition. Typical examples for $A$ are expressions involving spacial derivatives or multiplication operators on function spaces and \eqref{eqn:ACP} then models a partial differential equation. The classical operator theoretic approach to \eqref{eqn:ACP} is the method of strongly continuous semigroups. The latter provides tools to establish if $(A,\dom(A))$ is the generator of a so-called $\Cnull$-semigroup, i.e., if there exists a map $T\colon[0,\infty)\rightarrow L(X)$ satisfying the evolution property $T(0)=I_X$ and $T(t+s)=T(t)T(s)$ for $s$, $t\geq0$, such that all orbits $t\mapsto T(t)x$ are continuous, and such that $Ax=\lim_{t\rightarrow0}\frac{1}{t}(T(t)x-x)$ holds for all $x\in\dom(A)$. If this is the case, then $T$ provides, for every initial value $x_0\in X$, the unique (mild) solution of \eqref{eqn:ACP}. Having established the generator property, an extensive machinery becomes available to study qualitative and quantitative properties of the solutions; see the monographs \cite{EN, G2017, P1983} for details.

\medskip
Although the theory of $\Cnull$-semigroups has many applications, of course it does not apply to any evolution equation on any space. A well-known issue is that for some operators $(A,\dom(A))$ there is indeed a semigroup $T\colon[0,\infty)\rightarrow L(X)$ but its orbit maps fail to be continuous. A prototype example for this effect is the transport equation, i.e., $A=\frac{\dd}{\dd\xi}$, on the space $(\operatorname{C}_{\operatorname{b}}(\mathbb{R}),\|\cdot\|_{\infty})$ of bounded continuous functions on the real line. Whereas the classical theory works perfectly on the subspace of uniformly continuous functions, a quite natural expansion of the space destroys the strong continuity of the shift semigroup which gives the solutions to the corresponding Cauchy problem. A natural fix for this lack of continuity is to endow the space with a coarser topology. Here, for example the topology of pointwise convergence or compact convergence. The latter leads to the study of strongly continuous semigroups on locally convex vector spaces. Early results in this direction can be found, e.g., in Yosida's book \cite{Yosida}, or in papers by Miyadera \cite{Miyadera} from around 1960. Since then many fundamental $\Cnull$-results have been established analogously in this setting, see e.g., \cite{B1974, K1964, K1968,O1973}. Recent papers treat generation results for contraction semigroups, series representations, stability theory, interpolation and extrapolation spaces, and perturbation theory \cite{AJ2016, FJKW2014, GW2016, JWW2015, W2014, W2016}.

%K2009, Kr2016,

\medskip
Considering general locally convex spaces widens the class of function spaces that can be taken for $X$ a lot and covers also many non-normable spaces like those formed by test functions, Schwartz functions, or distributions. On the other hand the aforementioned example of the shift on $\operatorname{C}_{\operatorname{b}}(\mathbb{R})$ is indeed defined on a Banach space and the same holds true also for other examples, e.g., for Ornstein--Uhlenbeck processes or Feller processes, see \cite{LB2007}. The method of choice in these cases is the concept of so-called bi-continuous semigroups which was introduced by K\"{u}hnemund \cite{Ku} and amalgamates the norm topology and a weaker locally convex topology in a sense that we will make precise in Section \ref{SEC:1}.

\medskip
Since its outset in 2001, the theory of bi-continuous semigroups has grown and by now includes a Hille-Yosida type generation theorem, an approximation theory as well as a perturbation theory including inter- and extrapolation spaces \cite{AM2004,BPos,BF,BF2,ESF2006,F2004, F2004Semi}. In this paper we add another piece to the puzzle by establishing a Lumer--Phillips type generation theorem for contraction semigroups. This is in particular desirable since in the $\Cnull$-setting the Lumer--Phillips theorem, see \cite[Chapter II, Thm.~3.15]{EN}, is one of the most important tools to prove that a given operator is a generator. Moreover, a Lumer--Phillips theorem for (equicontinuous) semigroups on general locally convex spaces has been established recently by Albanese, Jornet \cite{AJ2016}. To demonstrate the applicability of our theorem we treat the prototype example from above, where we now are able to prove the generation property without using a priori knowledge on the semigroup. The latter was not possible beforehand. Our whole last chapter is then devoted to applications treating flows on infinite networks.

\section{Bi-continuous semigroups}\label{SEC:1}
Let us first recall some basic facts and properties of bi-continuous semigroups. The theory of the latter was inititated by K\"uhnemund \cite{Ku}. Her basic assumptions are actually strongly related to Saks spaces \cite{C1987}.

\begin{assumptions}\label{asp:bicontspace} 
Consider a triple $(X,\|\cdot\|,\tau)$ where $X$ is a Banach space, and
\begin{num}
\item $\tau$ is a locally convex Hausdorff topology coarser than the norm-topology on $X$, i.e., the identity map $(X,\|\cdot\|)\to(X,\tau)$ is continuous;
\item $\tau$ is sequentially complete on the $\left\|\cdot\right\|$-closed unit ball, i.e., every $\left\|\cdot\right\|$-bounded $\tau$-Cauchy sequence is $\tau$-convergent;
\item The dual space of $(X,\tau)$ is norming for $X$, i.e.,
\begin{equation}\label{eq:norm}
\|x\|=\sup_{\substack{\varphi\in(X,\tau)'\\\|\varphi\|\leq1}}{|\varphi(x)|},\quad x\in X.\end{equation}
\end{num}
We call the triple $(X,\|\cdot\|,\tau)$ also a \emph{bi-admissible space}.
\end{assumptions}

\begin{remark}Notice that a locally convex Hausdorff topology $\tau$ can be described by a fundamental system of seminorms $\Gamma$ in the sense of \cite[Chapter 22]{MV1992}. The set $\operatorname{cs}(X,\tau)$ of all continuous seminorms is always such a fundamental system.
%Notice, that every locally convex topology $\tau$ yields a family of continuous seminorms $\semis$ and vice versa, cf. \cite[Chapter II, Sect.~4]{Schaefer1971} or \cite[Thm.~1.36 \& 1.37]{Rudin}. If we want to stress the locally convex topology $\tau$ generated by $\semis$, we will use the notation $\semis_\tau$.
\end{remark}

Now we introduce bi-continuous semigroups as it was done by K\"uhnemund.

\begin{definition}\label{def:bicontsemi}
Let $(X,\|\cdot\|,\tau)$ be a bi-admissible space, i.e., the conditions of Assumption \ref{asp:bicontspace} are satisfied. We call a family of bounded linear operators $(T(t))_{t\geq0}$ on $X$ a \emph{bi-continuous semigroup} if
\begin{num}
\item $ T(t+s)=T(t)T(s)$ and $T(0)=I_X$ for all $s,t\geq 0$,
\item $(T(t))_{t\geq0}$ is strongly $\tau$-continuous, i.e. the map $\varphi_x:[0,\infty)\to(X,\tau)$ defined by $\varphi_x(t)=T(t)x$ is continuous for every $x\in X$,
\item $(T(t))_{t\geq0}$ has type $(M,\omega)$ for some $M\geq 1$ and $\omega\in \RR$, i.e., $\left\|T(t)\right\|\leq M\ee^{\omega t}$ for all $t\geq0$,
\item $(T(t))_{t\geq0}$ is locally-bi-equicontinuous, i.e., if $(x_n)_{n\in\NN}$ is a norm-bounded sequence in $X$ which is $\tau$-convergent to $0$, then also $(T(s)x_n)_{n\in\NN}$ is $\tau$-convergent to $0$ uniformly for $s\in[0,t_0]$ for each fixed $t_0\geq0$.
\end{num}
\end{definition}

Similarly to the case of $\Cnull$-semigroups, one defines the generator of a bi-continuous semigroup as follows.

\begin{definition}\label{def:BiGen}
Let $(T(t))_{t\geq0}$ be a bi-continuous semigroup on a bi-admissible space $(X,\|\cdot\|,\tau)$. The \emph{(infinitesimal) generator} of $(T(t))_{t\geq0}$ is the linear operator $(A,\dom(A))$ defined by
\[Ax:=\tau\!\shortminus\!\lim_{t\to0}{\frac{T(t)x-x}{t}}\] with domain 
\[\dom(A):=\Bigl\{x\in X:\ \tau\!\shortminus\!\lim_{t\to0}{\frac{T(t)x-x}{t}}\ \text{exists and} \ \sup_{t\in(0,1]}{\frac{\|T(t)x-x\|}{t}}<\infty\Bigr\}.\]
\end{definition}

In what follows, for a linear operator $(A,\dom(A))$, we will denote by $\rho(A)$ the resolvent set defined by
\begin{align}\label{eqn:ResolventSet}
\rho(A):=\left\{\lambda\in\CC:\ \lambda-A\ \text{is bijective}\right\}.
\end{align}
Moreover, for $\lambda\in\rho(A)$ we denote $R(\lambda,A):=(\lambda-A)^{-1}$, cf. \cite[Chapter IV, Def.~1.1]{EN}. In contrast to \cite{AJ2016}, we do not need to require that $R(\lambda,A)$ is bounded, since this follows directly by the closed graph theorem. As a matter of fact, the generator of a bi-continuous semigroup behaves almost like a $\Cnull$-semigroup generator. Let us recall some important properties of bi-continuous semigroups and their generators, cf. \cite[Sect.~1.2]{Ku}. %\cite[Prop.~1.16 \& Prop.~1.18]{KuPhD}.

\begin{theorem}\label{thm:BiContProp}
Let $(T(t))_{t\geq0}$ be a bi-continuous semigroup with generator $(A,\dom(A))$. Then the following hold:
\begin{abc}
\item The operator $A$ is bi-closed, i.e., whenever $x_n\stackrel{\tau}{\to}x$ and $Ax_n\stackrel{\tau}{\to}y$ and both sequences are norm-bounded, then $x\in\dom(A)$ and $Ax=y$.
\item The domain $\dom(A)$ is bi-dense in $X$, i.e., for each $x\in X$ there exists a norm-bounded sequence $(x_n)_{n\in\NN}$ in $\dom(A)$ such that $x_n\stackrel{\tau}{\to}x$.
\item For $x\in\dom(A)$ one has $T(t)x\in\dom(A)$ and $T(t)Ax=AT(t)x$ for all $t\geq0$.
\item For $t>0$ and $x\in X$ one has \begin{align}\int_0^t{T(s)x\ \dd s}\in\dom(A)\ \ \text{and}\ \ A\int_0^t{T(s)x\ \dd s}=T(t)x-x. \end{align}
\item For $\lambda>\omega$ one has $\lambda\in\rho(A)$ (thus $A$ is closed) and 
\begin{align}\label{eq:bicontlaplace}
R(\lambda,A)x=\int_0^{\infty}{\ee^{-\lambda s}T(s)x\ \dd s}{,\quad x\in X},
\end{align} where the integral is a $\tau$-improper Riemann integral, i.e., the limit
\[
\tau\!\shortminus\!\lim_{\alpha\to\infty}{\int_0^{\alpha}{\ee^{-\lambda s}T(s)x\ \dd s}},
\]
exists and the latter integral has to be understood as a $\tau$-Riemann integral.
\end{abc}
\end{theorem}

As in the case of strongly continuous semigroups on Banach spaces, there exists a Hille--Yosida generation type theorem, cf. \cite[Thm.~16]{Ku} or \cite[Thm.~5.6]{BF}. Here we will only formulate this theorem for bi-continuous contraction semigroups, i.e., $\left\|T(t)\right\|\leq1$ for all $t\geq0$, since these are the main objects in the context of the Lumer--Phillips theorem.

\begin{theorem}\label{thm:BiContHYCont}
Let $(X,\|\cdot\|,\tau)$ be a bi-admissible space, and let $(A,\dom(A))$ be a linear operator on the Banach space $X$. The following are equivalent:
\begin{abc}
\item The operator $(A,\dom(A))$ is the generator of a bi-continuous contraction semigroup $(T(t))_{t\geq0}$.
\item The operator $(A,\dom(A))$ satisfies $(0,\infty)\subseteq\rho(A)$ and
\begin{align}\label{eqn:HYOp}
\|R(\lambda,A)^{n}\|\leq\frac{1}{\lambda^n}
\end{align}
for all $n\in\NN$ and for all $\lambda>0$. Moreover, $A$ is bi-densely defined and the family
\begin{align}\label{eq:resbiequi}
\bigl\{\lambda^nR(\lambda,A)^n:\ n\in\NN,\ \lambda>0\bigr\}
\end{align}
 is bi-equicontinuous, meaning that for each norm bounded $\tau$-null sequence $(x_m)_{m\in\NN}$ one has $\lambda^nR(\lambda,A)^nx_m\stackrel{\tau}{\rightarrow} 0$ in $\tau$ uniformly for $n\in\NN$ and $\lambda>0$ as $m\to\infty$.
\end{abc}
In this case, we have the Euler formula
\begin{align}\label{eqn:Euler}
T(t)x:=\tau\!\shortminus\!\lim_{m\to\infty}\left(\frac{m}{t}R\left(\frac{m}{t},A\right)\right)^mx,\quad x\in X.
\end{align}
\end{theorem}

\section{A Lumer--Phillips generation type theorem}
The original Lumer--Phillips generation theorem for strongly continuous semigroups on Banach spaces goes back to famous work of Lumer and Phillips from around 1959--61. Below we follow the monograph \cite[Chapter II, Def.~3.13]{EN} and recall the following central definition.

\begin{definition}\label{def:Dissip}
A linear operator $(A,\dom(A))$ on a Banach space $X$ is called \emph{dissipative} if
\[
\left\|(\lambda-A)x\right\|\geq\lambda\left\|x\right\|,
\]
for all $\lambda>0$ and $x\in\dom(A)$.
\end{definition}

In our setting we need to control not only the norm but also the locally convex topology in a way similar to Definition \ref{def:Dissip}. The definition is motivated by the work of Albanese and Jornet \cite{AJ2016} and will also be crucial for the generation theorem in the bi-continuous setting. 

\begin{definition}\label{def:taudiss}
Let $(A,\dom(A))$ be an operator on a bi-admissible space $(X,\|\cdot\|,\tau)$. We call this operator \emph{bi-dissipative} if there exists a fundamental system of seminorms $\Gamma$ generating the topology $\tau$ such that 
\begin{align}\label{eqn:taudiss}
p((\lambda-A)x)\geq \lambda p(x),\quad \lambda>0,\;p\in\Gamma,\;x\in\dom(A)
\end{align}
and 
\begin{align}\label{eqn:taudiss-2}
\left\|x\right\|=\sup_{p\in\Gamma}{p(x)},\quad x\in X.
\end{align}
\end{definition}

\begin{remark}\label{rem:taudiss}
Let us summarize some observations regarding the bi-dissipative operators introduced in Definition \ref{def:taudiss}.
\begin{iiv}
	\item Since $(X,\|\cdot\|,\tau)$ is supposed to be a bi-admissible space, cf. Assumption \ref{asp:bicontspace}, the dual space of $(X,\tau)$ is norming for $X$. By \cite[Rem.~5.2(4)]{BF} this is the case if and only if there exists a fundamental system of seminorms $\Gamma$ such that
	\begin{align}\label{eqn:normingProp}
	\left\|x\right\|=\sup_{p\in\Gamma}{p(x)},\quad x\in X.
	\end{align}
	In view of \eqref{eqn:taudiss} one directly obtains that each bi-dissipative operator also satisfies 
	\[
	\left\|(\lambda-A)x\right\|\geq\lambda\left\|x\right\|,
	\]
	for all $x\in\dom(A)$, i.e., bi-dissipative operators on bi-admissible spaces are also dissipative in the sense of Definition \ref{def:Dissip}.
	\item Definition \ref{def:taudiss} is motivated by \cite[Def.~3.9]{AJ2016}, where Albanese and Jornet called an operator $(A,\dom(A))$ on a locally convex space $(X,\tau)$ \emph{$\Gamma$-dissipative} if \eqref{eqn:taudiss} holds. A bi-dissipative operator is thus $\Gamma$-disspative in the sense of \cite{AJ2016} if we pick the fundamental system of seminorms $\Gamma$ as in Definition \ref{def:taudiss}. Notice that $\Gamma$-dissipativity depends on the choice of the fundamental system, see \cite[Rem.~3.10]{AJ2016}.
\end{iiv}	
\end{remark}

Let us collect some basic properties of bi-dissipative operators, cf. \cite[Prop.~3.11]{AJ2016}, which will be useful for the following.

\begin{proposition}\label{prop:PropDissOp}
Let $(A,\dom(A))$ be a bi-dissipative operator with respect to a fundamental system of seminorms $\Gamma$ on a bi-admissible space $(X,\|\cdot\|,\tau)$. Then the following assertions are true:
\begin{iiv}
	\item $\lambda-A$ is injective for all $\lambda>0$. Moreover, one has that
	\begin{align}\label{eqn:taudissRes}
	p(R(\lambda,A)x)\leq\frac{1}{\lambda}p(x),
	\end{align}
	for all $\lambda>0$, $p\in\Gamma$ and $x\in\mathrm{Ran}(\lambda-A)$.
	\item $\lambda-A$ is surjective for some $\lambda>0$ if and only if it is surjective for all $\lambda>0$. In addition, one has that $\left(0,\infty\right)\subseteq\rho(A)$.
\end{iiv}
\end{proposition}

\begin{remark}\label{rem:DissNormResol}
Making use of Remark \ref{rem:taudiss}(i) we also get that $\left\|R(\lambda,A)x\right\|\leq\frac{1}{\lambda}\left\|x\right\|$ for all $\lambda>0$ and $x\in\mathrm{Ran}(\lambda-A)$. Hence, the second assertion of Proposition \ref{prop:PropDissOp} is also a direct consequence of \cite[Chapter II, Prop.~3.14(ii)]{EN}.
\end{remark}
%We notice that if $(A,\dom(A))$ is $\tau$-dissipative on a bi-admissible space $(X,\|\cdot\|,\tau)$ one also obtains that $\lambda-A$ is surjective for some $\lambda>0$ is and only if it is surjective for all $\lambda>0$ and that, in addition, one has that $\left(0,\infty\right)\subseteq\rho(A)$. To see this, observe that by Remark \ref{rem:taudiss}(i) we conclude that the operator $\lambda-A$ is bijective and $\mathrm{Ran}(R(\lambda,A))=\dom(A)$. With \eqref{eqn:taudissRes} we obtain that $p(R(\lambda,A)x)\leq\frac{1}{\lambda}p(x)$ for all $x\in X$ and $p\in\semis_\tau$. Notice that by a combination of Remark \ref{rem:taudiss} and \cite[Chapter II, Prop.~3.14(iii)]{EN} it follows that $(A,\dom(A))$ is a $\left\|\cdot\right\|$-closed operator. Hence by \cite[Chapter IV, Prop.~1.3]{EN} we obtain the result.

\medskip
Let us now state our first result.

\begin{theorem}\label{thm:LP1}
Let $(A,\dom(A))$ be a bi-dissipative, bi-densely defined and closed operator on a bi-admissible space $(X,\|\cdot\|,\tau)$. Then the following statements are equivalent:
\begin{abc}
	\item $(A,\dom(A))$ generates a bi-continuous contraction semigroup on $X$.
	\item $\lambda-A$ is surjective for some $\lambda>0$.
\end{abc}
\end{theorem}

\begin{proof}
For both implications, we make use of the Hille--Yosida generation type theorem for bi-continuous contraction semigroups, cf. Theorem \ref{thm:BiContHYCont}.\\
(a)$\Rightarrow$(b): Assume, that $(A,\dom(A))$ generates a bi-continuous contraction semigroup on $X$. By Theorem \ref{thm:BiContHYCont} we conclude that $(0,\infty)\subseteq\rho(A)$. In particular, making use of \eqref{eqn:ResolventSet} we observe that $\lambda-A$ is surjective for all and hence for some $\lambda>0$.\\
(b)$\Rightarrow$(a): In order to show that $(A,\dom(A))$ generates a bi-continuous contraction semigroup, by Theorem \ref{thm:BiContHYCont} it suffices to show that $\left\|R(\lambda,A)^n\right\|\leq\frac{1}{\lambda^n}$ for all $n\in\NN$ and $\lambda>0$ and that the family of operators $\left\{\lambda^nR(\lambda,A)^n:\ n\in\NN, \lambda>0\right\}$ is bi-equicontinuous. By assumption, the operator $(A,\dom(A))$ is bi-dissipative. Hence, by an application of Proposition \ref{prop:PropDissOp} we conclude %and there exists $\lambda>0$ such that $\lambda-A$ is surjective. Hence, from Proposition \ref{prop:PropDissOp}(ii) we deduce that actually $\lambda-A$ is surjective for all $\lambda>0$. In addition, by Proposition \ref{prop:PropDissOp}(i), the operator $\lambda-A$ is also injective for all $\lambda>0$. Therefore, we conclude 
that the operator $\lambda-A$ is bijective for all $\lambda>0$. By the definition of the resolvent set \eqref{eqn:ResolventSet} this implies that $(0,\infty)\subseteq\rho(A)$. Moreover, we obtain that \eqref{eqn:taudissRes} holds which implies that $\left\{\lambda^nR(\lambda,A)^n:\ n\in\NN, \lambda>0\right\}$ is bi-equicontinuous since $p(\lambda^nR(\lambda,A)^nx)\leq p(x)$ for all $p\in\Gamma$, $\lambda>0$ and $x\in X$. By Remark \ref{rem:DissNormResol} we obtain that $(A,\dom(A))$ satisfies $\left\|R(\lambda,A)x\right\|\leq\frac{1}{\lambda}\left\|x\right\|$ for all $\lambda>0$ and $x\in\mathrm{Ran}(\lambda-A)=X$. As a consequence, we obtain that $\left\|R(\lambda,A)^n\right\|\leq\frac{1}{\lambda^n}$ for all $n\in\NN$ and $\lambda>0$. This proves the statement.
\end{proof}

Let us mention that for the implication (a)$\Rightarrow$(b) the bi-dissipativity is not needed. However, the assumption that $(A,\dom(A))$ generates a bi-continuous contraction semigroup does not imply in general that $(A,\dom(A))$ is bi-dissipative. To see this, assume that $(A,\dom(A))$ generates a bi-continuous contraction semigroup. By Theorem \ref{thm:BiContHYCont} one gets that \eqref{eqn:HYOp} holds which implies that $(A,\dom(A))$ is dissipative in the sense of Definition \ref{def:Dissip}. In addition, Theorem \ref{thm:BiContHYCont} shows that $\left\{\lambda^nR(\lambda,A)^n:\ n\in\NN, \lambda>0\right\}$ is bi-equicontinuous. Nevertheless, this does not imply that $(A,\dom(A))$ is bi-dissipative. We will illustrate this with an example. 

\begin{example}\label{exm:CounterEx}
Consider the space $X:=\BC\left(\left(\infty,0\right]\right)$ equipped with the supremum norm and the compact-open topology $\tau_\mathrm{co}$. Recall, that $\tau_\mathrm{co}$ is a locally convex topology generated by the seminorms $p_n$, $n\in\mathbb{N}$, where $p_n(f):=\sup_{x\in [-n,0]}{\left|f(x)\right|}$, $f\in X$. On the space $X$ we now consider the right-translation semigroup defined by $(T(t)f)(x)=f(x-t)$ for $t\geq0$, $f\in X$ and $x\in\left(-\infty,0\right]$. This semigroup becomes a bi-continuous contraction semigroup on $X$ with respect to the compact-open topology. Assume that its generator is bi-dissipative, i.e., there exists a fundamental system of seminorms $\Gamma$ such that $q(R(\lambda,A)f)\leq\frac{1}{\lambda}q(f)$ for all $\lambda>0$, $q\in\Gamma$, $f\in X$. Since the two fundamental systems of seminorms induce the same topology, we find firstly for $p_1$ a seminorm $q\in\Gamma$ and $C>0$ and then secondly a seminorm $p_n$ and $K>0$ such that $p_1(f)\leq Cq(f)\leq Kp_n(f)$ for $f\in X$. Now define a function $f\in\BC\left(\left(\infty,0\right]\right)$ as in the following picture.

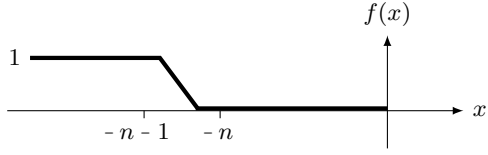
\begin{center}
\begin{tikzpicture}[scale=1]

  \draw[-latex] (-3,0) -- (3,0) node[right] {\small$x$};
  \draw[-latex] (2,-0.5) -- (2,1.) node[above] {\small$f(x)$};
  \draw [ultra thick, draw=black] (-2.7,0.7) -- (-0.978,0.7);
    \draw [ultra thick, draw=black] (-1,0.71) -- (-0.495,0.03);
    \draw [ultra thick, draw=black] (-0.52,0.03) -- (2,0.03);
   \node[] at (-2.9,0.7) {\small$1$};
      \draw[] (-0.2,-0.1) -- (-0.2,0) node[below] {};
      \draw[] (-1.2,-0.1) -- (-1.2,0) node[below] {};
     \node[] at (-0.23,-0.3) {\small$\shortminus\,n$};
        \node[] at (-1.3,-0.28) {\small$\shortminus\,n\shortminus1$};
\end{tikzpicture}
\end{center}

\medskip

Then, $p_{n}(f)=0$ and therefore $\frac{1}{\lambda}q(f)=0$. For the resolvent we estimate the following
\begin{align*}
p_1(R(\lambda,A)f)&=p_1\left(\int_0^\infty{\ee^{-\lambda t}T(t)f\ \dd{t}}\right)\geq p_1\left(\int_{n+1}^\infty{\ee^{-\lambda t}T(t)f\ \dd{t}}\right)\\
&=\sup_{x\in [-1,0]}{\left|\int_{-\infty}^{x-n-1}{\ee^{-\lambda(x-s)}f(s)}\ \dd{s}\right|}\\
&=\sup_{x\in [-1,0]}\left|\int_{-\infty}^{x-n-1}{\ee^{-\lambda(x-s)}}\ \dd{s}\right|\\
&=\sup_{x\in [-1,0]}{\left|\frac{1}{\lambda}\ee^{-\lambda(n+1)}\right|}>0
\end{align*}
which implies that $q(R(\lambda,A)f)>0$.
\end{example}	

One of the prototype examples, where the classical Lumer-Phillips theorem can be applied, is the operator $Af=-f'$ on the space $\operatorname{C}[a,b]$, see \cite[Chapter II, Ex.~3.19]{EN}. Below we show that our bi-continuous Lumer--Phillips theorem applies to the same operator but considered on the space $\BC\left(\left[0,\infty\right)\right)$. Of course, in this case, it is well-known and can easily be checked directly, that $A$ generates the left shift semigroup. Indeed, the latter is one of the typical examples for bi-continuous semigroups which are not $\Cnull$. Example \ref{ex:TranslationSemi} illustrates however that with Theorem \ref{thm:LP1} one can conclude the generator property \emph{without knowing in advance that the semigroup will be the shift}.

\begin{example}\label{ex:TranslationSemi}
Let us consider the space $X:=\BC\left(\left[0,\infty\right)\right)$ equipped with the supremum norm as well as the compact-open topology $\tau_\mathrm{co}$, see Example \ref{exm:CounterEx}. We consider the operator $(A,\dom(A))$ defined by
\[
Af:=-f',\quad \dom(A):=\left\{f\in\BC^1\left(\left[0,\infty\right)\right):\ f(0)=0\right\}.
\] 
Let us show that the operator $(A,\dom(A))$ is bi-dissipative in the sense of Definition \ref{def:taudiss}. The resolvent of $(A,\dom(A))$ can be determined explicitly by
\[
(R(\lambda,A)f)(x)=\int_0^x{\ee^{\lambda(t-x)}f(t)\ \dd{t}}=\ee^{-\lambda x}\int_0^x{\ee^{\lambda t}f(t)\ \dd{t}},\quad x\in\left[0,\infty\right),\ f\in\BC\left(\left[0,\infty\right)\right),\ \lambda>0,
\]
by solving the differential equation $f'=\lambda f-g$. This shows that $\lambda-A$ is surjective, since indeed $R(\lambda,A)f$ defines a continuously differentiable function satisfying $f(0)=0$ which is bounded. To show that $(A,\dom(A))$ is bi-dissipative, in the sense of Definition \ref{def:taudiss}, we use the fundamental system of seminorms $(p_n)_{n\in\NN}$ with $p_n(f):=\sup_{x\in\left[0,n\right]}{\left|f(x)\right|}$. Let $n\in\NN$ be arbitrary and observe that for $x\in\left[0,n\right]$ one has
\[
\lambda\left|(R(\lambda,A)f)(x)\right|\leq\lambda\ee^{-\lambda x}\int_0^x{\ee^{\lambda s}\left|f(s)\right|\ \dd{s}}\leq\lambda\ee^{-\lambda x}\int_0^x{\ee^{\lambda s}p_n(f)\ \dd{s}}=(1-\ee^{-\lambda x})p_n(f)\leq p_n(f),
\]
showing that $(A,\dom(A))$ is a bi-dissipative operator.
\end{example}

After this first `toy example' let us consider now the heat equation.

\begin{example}\label{ex:heatsemi} Let $X:=\BC(\RR)$, again endowed with the supremum norm and the topology of uniform convergence on compact sets, and consider the one-dimensional Laplacian, i.e., $(A,\dom(A))$ defined by
\[
Af:=f'',\quad\dom(A):=\BC^2(\RR).\vspace{5pt}
\]
If we put $n=2$, $\lambda=1$ and pick $f\in\BC^2(\RR)$ such that $f(x)=x^2$ for $x\in[-2,2]$, then $p_n((\lambda-A)f)=2$ but $\frac1\lambda p_n(f) =4$, and we see that $A$ is not bi-dissipative with respect to the fundamental system of seminorms mentioned in Example \ref{ex:TranslationSemi}. Proposition \ref{prop:EquiContDissip} below will even show that $A$ cannot be bi-dissipative for any fundamental system: The operator $A$ indeed generates a bi-continuous semigroup which is however not locally equicontinuous, see \cite[Ex.~6]{Ku}, and can therefore not satisfy Proposition \ref{prop:EquiContDissip}(a).

\end{example}

\begin{remark}
In Example \ref{ex:heatsemi} no initial conditions are necessary since the domain of the functions is the whole real line $\RR$. However, when restricting to unbounded domains $\Omega\subseteq\RR^n$ one can add Dirichlet boundary conditions in order to obtain a generator of a bi-continuous semigroup on the space $\mathrm{C_{b,0}}(\overline{\Omega}):=\left\{f\in\BC(\overline{\Omega}):\ f_{\left|\partial\Omega\right.}=0\right\}$, cf. \cite[Sect.~2.4]{FaPhD}.
\end{remark}

We will give another example that treats differential equations on infinite networks in Section \ref{SEC:3}. Before we come to that we will however first establish a Lumer--Phillips variant that characterizes when the closure of an operator is a generator in the bi-continuous setting.

\begin{proposition}\label{prop:EquiContDissip}
Let $(A,\dom(A))$ be a bi-densely defined and closed operator on a bi-admissible space $(X,\|\cdot\|,\tau)$. Then the following statements are equivalent:
\begin{abc}
	\item $(A,\dom(A))$ generates a bi-continuous contraction semigroup $(T(t))_{t\geq0}$ and there exists a fundamental system of seminorms $\Gamma$ which is firstly norming, i.e., \eqref{eqn:normingProp} holds, and secondly such that $p(T(t)x)\leq p(x)$ holds for all $t\geq0$, $p\in\Gamma$ and $x\in X$.
	\item $(A,\dom(A))$ is bi-dissipative and $\lambda-A$ is surjective for some $\lambda>0$.
\end{abc}
\end{proposition}

%\textcolor{green}{Woraus folgt die Absch\"atzung $p(T(t)x)\leq p(x)$ in (b)\,$\Rightarrow$\,(a)?}
\begin{proof}
In order to show the implication (a)$\Rightarrow$(b) we make use of Theorem \ref{thm:LP1}. It suffices to show that $(A,\dom(A))$ is bi-dissipative. To do so, we observe that for $\lambda>0$, $p\in\Gamma$ and $x\in X$ one has that
\[
p(R(\lambda,A)x)=p\left(\int_0^\infty{\ee^{-\lambda t}T(t)\ \dd{t}}\right)\leq\int_0^{\infty}{\ee^{-\lambda t}p(T(t)x)\ \dd{t}}\leq\int_0^{\infty}{\ee^{-\lambda t}p(x)\ \dd{t}}=\frac{1}{\lambda}p(x),
\]
showing that $(A,\dom(A))$ is bi-dissipative (with respect to the fundamental system of seminorms $\Gamma$). Notice, that we use the fact, that the resolvent $R(\lambda,A)$ can be expressed by means of the Laplace transform, cf. Theorem \ref{thm:BiContProp}(e), and that the first inequality follows from the continuity of the seminorms. For the converse implication (b)$\Rightarrow$(a) we use the Euler formula \eqref{eqn:Euler} as well as the resolvent estimate \eqref{eqn:taudissRes}.
\end{proof}

The following remark shows that the asumptions of Proposition \ref{prop:EquiContDissip}(a) hold true up to "rescaling", for every bi-continuous semigroup which is equicontinuous with respect to the locally convex topology.

%\textcolor{green}{Der letzte Satz in der Bemerkung unten gilt nur wenn man gleichstetig voraussetzt, oder? Wenn dem so ist, sollte man vielleicht lieber schreiben, dass man in 2.10(a) `gleichstetig' schreiben kann.}

\begin{remark}
Let $(T(t))_{t\geq0}$ be an equicontinuous and bi-continuous semigroup. Following \cite[Thm.~1.2.7(e) \& Sect.~1.3(b)]{FaPhD} we may assume that the semigroup is bounded by rescaling, i.e., one considers the operator semigroup $(\ee^{-\omega_0}T(t))_{t\geq0}$ where $\omega_0$ denotes the growth bound of the original operator semigroup $(T(t))_{t\geq0}$, cf. \cite[Chapter 1, Def.~5.6]{EN}. Then there exists a norming fundamental system of seminorms $\Gamma$, i.e., the fundamental system of seminorms $\Gamma$ satisfies \eqref{eqn:normingProp}. Using the equicontinuity (which is independent of the chosen fundamental system) we can firstly proceed as in \cite[Rem.~2.2]{ABR2013} and define a new fundamental system of seminorms $\widehat{\Gamma}:=\left\{\widehat{p}:\ p\in\Gamma\right\}$, where
\[
\widehat{p}(x):=\sup_{t\geq0}p(T(t)x),\quad x\in X,\ p\in\Gamma,
\]
which induces the same topology as $\Gamma$ and satisfies $\widehat{p}(T(t)x)\leq\widehat{p}(x)$ for all $p\in\widehat{\Gamma}$, $x\in X$ and $t\geq0$. Secondly, we can follow \cite[Sect.~1.3(b)]{FaPhD} and put
\[
\opnorm{x}:=\sup_{t\geq0}\left\|T(t)x\right\|,\quad x\in X.
\]
Then $\|\cdot\|$ and $\opnorm{\cdot}$ are equivalent and $\widehat{\Gamma}$ is norming for $\opnorm{\cdot}$, since
\[
\sup_{\widehat{p}\in\widehat{\Gamma}}{\widehat{p}(x)}=\sup_{p\in\Gamma}\,\sup_{t\geq0}p(T(t)x)=\sup_{t\geq0}\,\sup_{p\in\Gamma}{p(T(t)x)}=\sup_{t\geq0}{\left\|T(t)x\right\|}=\opnorm{x}
\]
holds.
\end{remark}

We want to go a step further by weakening the assumptions in Theorem \ref{thm:LP1}. To do so, we recall the notion of closable operators. Here we have to distinguish between closable operators on Banach spaces and closable operators on locally convex spaces.

\begin{definition}
Let $(A,\dom(A))$ be a linear operator on a Banach space $X$. We call the operator \emph{closable} if for every sequence $(x_n)_{n\in\NN}$ in $\dom(A)$ with $x_n\to0$ and $Ax_n\to y$ one has $y=0$.
\end{definition}

\begin{definition}
Let $(A,\dom(A))$ be a linear operator on a locally convex space $(X,\tau)$. We call the operator \emph{$\tau$-closable} if for every net $(x_\alpha)_{\alpha\in A}$ in $\dom(A)$ with $x_\alpha\to0$ and $Ax_\alpha\to y$ one has $y=0$.
\end{definition} 

The difficulty here is to combine these two notions for the case of bi-continuous semigroups. This can be done by means of the so-called mixed topology which was introduced by Wiweger \cite{W1961} for general topological spaces. We will here however restrict ourselves to the simpler situation of bi-admissible spaces. In this case according to \cite[Thm.~3.1.1]{W1961} and \cite[Sect.~A.1]{FaPhD} the mixed topology is given by the fundamental system of seminorms 
\[
\widetilde{\Gamma}:=\left\{\widetilde{p}_{(a_n),(p_n)}:\ (p_n)_{n\in\mathbb{N}}\subseteq\Gamma,\ (a_n)_{n\in\NN}\in\mathrm{c}_0,\ a_n\geq0\right\}
\]
where
\[
\widetilde{p}_{(a_n),(p_n)}(x):=\sup_{n\in\NN}{a_np_n(x)},\quad x\in X,
\]
and will be denoted by $\gamma:=\gamma(\tau,\left\|\cdot\right\|)$. It is obvious, that $\tau\subseteq\gamma\subseteq\left\|\cdot\right\|$ holds and by \cite[Thm.~2.3.1]{W1961} one gets that convergent sequences with respect to the mixed topology are precisely the $\left\|\cdot\right\|$-bounded and $\tau$-convergent sequences relevant in our setting. Moreover, the classes of bi-continuous semigroups and $\gamma$-strongly continuous and locally sequentially $\gamma$-equicontinuous semigroups coincide, cf. \cite[Prop.~1.6]{F2011Cz}.

\medskip

We will now apply \cite[Thm.~3.14]{AJ2016} to the mixed topology $\gamma$.
%We can summarize our observations in the following theorem.

\begin{theorem}\label{thm:LP2}
Let $(A,\dom(A))$ be a bi-dissipative and bi-densely defined operator on a bi-admissible space $(X,\|\cdot\|,\tau)$. Assume that $X$ is complete with respect to the mixed topology $\gamma$. If $\mathrm{Ran}(\lambda-A)$ is bi-dense in $X$ for some $\lambda>0$, then the $\gamma$-closure $(\overline{A},\dom(\overline{A}))$ of $(A,\dom(A))$ generates a bi-continuous contraction semigroup on $X$.
\end{theorem}

\begin{proof}
Firstly, we observe that if $(A,\dom(A))$ is bi-dissipative on a bi-admissible space $(X,\|\cdot\|,\tau)$, then the operator $(A,\dom(A))$ is $\widetilde{\Gamma}$-dissipative in the sense of \cite[Def.~3.9]{AJ2016}. To see this, let $(a_n)_{n\in\NN}$ be an arbitrary sequence with $a_n\to0$ and $(p_n)_{n\in\NN}$ also be arbitrary in $\Gamma$. By the assumption that $(A,\dom(A))$ is bi-dissipative we have $p((\lambda-A)x)\geq \lambda p(x)$ for all $p\in\Gamma$, $\lambda>0$ and $x\in\dom(A)$. Hence, we obtain
\[
\widetilde{p}_{(a_n,p_n)}((\lambda-A)x)=\sup_{n\in\NN}{a_np_n((\lambda-A)x)}\geq\sup_{n\in\NN}{\lambda a_np_n(x)}=\lambda\widetilde{p}_{(a_n,p_n)}(x),
\]
for all $x\in\dom(A)$, showing that $(A,\dom(A))$ is $\widetilde{\Gamma}$-dissipative with respect to a fundamental system of seminorms $\Gamma$. Furthermore, if $\mathrm{Ran}(\lambda-A)$ is bi-dense, then it is also $\gamma$-dense since by \cite[Lemma~A.1.2]{FaPhD} bi-density coincides with $\gamma$-sequential density. Moreover, $\gamma$-sequentially dense operators are always $\gamma$-dense. Now, we are able to apply \cite[Thm.~3.14]{AJ2016}. In fact, under the assumption that $X$ is complete with respect to the mixed topology $\gamma$, we conclude by \cite[Thm.~3.14]{AJ2016} that the $\gamma$-closure $(\overline{A},\dom(\overline{A}))$ of $(A,\dom(A))$ generates an equicontinuous strongly continuous semigroup on the locally convex space $(X,\gamma)$. By \cite[Prop.~A.1.3]{FaPhD} we get that this semigroup is a bi-continuous contraction semigroup.
\end{proof}

In general completeness can of course be lost if one moves over to a coarser topology. The completeness assumption in Theorem \ref{thm:LP2} is however met in many important examples:

\begin{remark}\label{rem:MixedTopProp}
\begin{iiv}
	\item Equip the space $\BC(\RR)$ of bounded continuous functions on the real line with the sup-norm $\left\|\cdot\right\|_\infty$ and the compact-open topology $\tau_{\mathrm{co}}$. Then $\BC(\RR)$ is complete with respect to the associated mixed topology $\gamma(\tau_{\mathrm{co}},\left\|\cdot\right\|_\infty)$, see for example \cite[Prop.~2.2(b)]{GK2001}.
	\item The space $\LLL(E)$ of bounded linear operators on a Banach space $E$ together with the operator norm $\left\|\cdot\right\|_{\LLL(E)}$ and the strong operator topology $\tau_{\mathrm{sot}}$ is also complete with respect to the mixed topology $\gamma(\tau_{\mathrm{sot}},\left\|\cdot\right\|_{\LLL(E)})$, cf. \cite[Chapter I, Prop.~1.14]{C1987}.
\end{iiv}	
\end{remark}

\section{Example: Bi-continuous semigroups for flows on infinite networks}\label{SEC:3}
\subsection{An alternative characterization of bi-dissipative operators}
In this section, we want to characterize bi-dissipative operators in an alternative way. In the case of $\Cnull$-semigroups, this is done by the so-called duality set, cf. \cite[Chapter II, Prop.~3.23]{EN}. For bi-continuous semigroups, we will make use of \cite[Prop.~4.2]{AJ2016} and characterize bi-dissipative operators by means of the so-called subdifferential which can be seen as the locally convex equivalent to the duality set. For $p\in\operatorname{cs}(X)$ and $x\in X$ one defines the set
\[
\mathcal{J}(x,p):=\left\{\varphi\in X':\ \mathrm{Re}\left\langle y,\varphi\right\rangle\leq p(y)\ \forall y\in X,\ \left\langle x,\varphi\right\rangle=p(x)\right\}.
\]
By the Hahn--Banach theorem $\mathcal{J}(x,p)$ is non-empty, see also \cite[Sect.~4]{AJ2016}. The following result characterizes bi-dissipative operators by means of the set $\mathcal{J}(x,p)$, cf. \cite[Prop.~4.2]{AJ2016}.

\begin{proposition}\label{prop:DissSubDiff}
Let $(A,\dom(A))$ be an operator on a bi-admissible space $(X,\|\cdot\|,\tau)$. Then the following assertions are equivalent:
\begin{abc}
	\item $(A,\dom(A))$ is bi-dissipative.
	\item There exists a fundamental system of seminorms $\Gamma$ satisfying \eqref{eqn:normingProp} and such that for all $p\in\Gamma$ and $x\in\dom(A)$ there exists $\varphi\in\mathcal{J}(x,p)$ such that $\mathrm{Re}\left\langle Ax,\varphi\right\rangle\leq0$.
\end{abc}
\end{proposition}

In the previous Proposition \ref{prop:DissSubDiff} we added the bi-admissibility for the sake of completeness and to stay self-contained. However, only the locally convex topology plays an actual role in that statement. 

\subsection{Flows on infinite networks}
We will use Proposition \ref{prop:DissSubDiff} in order to improve the result \cite[Thm.~3.17]{BK2019} concerning bi-continuous semigroups for flows in infinite networks. In fact, we show that the generation result mentioned there also holds for infinite networks and that the restriction to finite networks is not longer necessary.

\medskip

For the general set-up for flows on networks we refer to the work of Kramar Fijav\v{z} and Dorn et al. \cite{KS2005,DKNR2010,D2008}. Even though they discuss everything in the world of $C_0$-semigroups, the structure and notation is still the same for the bi-continous case. Notice that their work makes use of the Banach spaces $X:=\mathrm{L}^1\left(\left[0,1\right],\CC^n\right)$ or $X:=\mathrm{L}^1\left(\left[0,1\right],\ell^1\right)$, whereas in the bi-continuous setting one considers $X:=\mathrm{L}^\infty\left(\left[0,1\right],\ell^1\right)$. For the sake of completeness we will recall some important details from \cite{BK2019}.

\medskip 
Let $G=(V,E)$ be a graph with the set of vertices $V=\left\{\vv_i\mid i\in I\right\}$ and the set of {directed edges} $E=\left\{\ee_j\mid  j\in J\right\}\subseteq V\times V$ for some countable sets $I,J\subseteq\NN$. For a directed edge $\ee=(\vv_i,\vv_k)$ we call $\vv_i$ the tail and  $\vv_k$ the {head} of $\ee$. Moreover, the edge $\ee$ is an {outgoing edge} of the vertex $\vv_i$ and an {incoming edge} for the vertex $\vv_k$. In what follows, we assume that the graph $G$ is {simple} and \emph{locally finite}, i.e., there are no loops or multiple edges and each vertex only has finitely many {incident} edges. The graph $G$ is supposed to be {weighted}, that is we are giving numbers $0 \le  w_{ij} \le 1$ such that
\begin{equation}\label{eqn:noAbsorb}
\sum_{i\in J}{w_{ij}}=1 \text{ for all } j\in J.
\end{equation}
The structure of a graph can be described by its incidence and/or adjacency matrices. We shall only use the so-called {weighted (transposed) adjacency matrix $\mathbb{B} = (\mathbb{B}_{ij})_{i,j\in J}$ of the line graph}  defined as
\begin{equation}\label{eqn:adjMat}
\mathbb{B}_{ij}:=\begin{cases}
w_{ij}& \text{ if }\stackrel{\ee_j}{\longrightarrow}\vv\stackrel{\ee_i}{\longrightarrow},\\
0 & \text{ otherwise}.
\end{cases}
\end{equation}
By \eqref{eqn:noAbsorb}, the matrix $\mathbb{B}$ is column stochastic and defines a bounded positive operator on $\ell^1$ with spectral radius $r(\mathbb{B})= \|\mathbb{B}\|=1$. Let us identify every edge of our graph with the unit interval,  $\ee_j\equiv\left[0,1\right]$ for each $j\in J$, and parametrize it contrary to its direction, so that it is assumed to have its tail at the endpoint $1$ and its head at the endpoint $0$. For simplicity we use the notation $\ee_j(1)$ and $\ee_j(0)$ for the tail and the head, respectively. In this way we obtain a {metric graph}. On such an infinite network, we model a transport process (or a flow) along the edges of the corresponding metric graph $G$.  The distribution of material along  edge $\ee_j$ at time $t \ge0$ is described by the function $u_j(x,t)$ for $x\in [0, 1]$. The material is transported along edge $\ee_j$ with constant velocity $c_j>0$ {and is absorbed according to the absorption rate $q_j(x)$.} We assume that {$q\in\L^{\infty}\left(\left[0,1\right],\ell^{\infty}\right)$ and }
\begin{equation}\label{as:c-lim}
0< c_{\min} \le c_j \le c_{\max} < \infty
\end{equation}
for all $j\in J$.  Let  $C:=\diag(c_j)_{j\in J}$ be a diagonal velocity matrix and define another weighted adjacency matrix of the line graph by 
\begin{equation*}\mathbb{B}^C:=C^{-1}\mathbb{B} C.\end{equation*}
In the vertices the material gets redistributed according to some prescribed rules. This is modelled by the boundary conditions making use of the adjacency matrix $\mathbb{B}^C$. The flow process on $G$ is thus given by the following infinite system of equations
\begin{align}\label{eqn:F}
\begin{cases}
\frac{\partial}{\partial t}u_j(x,t)=c_j\frac{\partial}{\partial x}u_j(x,t) {+ q_j(x)u_j(x,t)},&\quad x\in\left(0,1\right),\ t\geq0,\\
u_j(1,t)= \sum_{k\in J} \mathbb{B}^C_{jk} u_k(0,t),&\quad t\geq0,\\
u_j(x,0)= f_j(x),&\quad x\in\left(0,1\right),
\end{cases}
\end{align}
for every $j\in J$, where $f_j(x)$ are the initial distributions along the edges. It has been proved in \cite[Sect.~3.1]{BK2019} that \eqref{eqn:F} can actually been rewritten as an abstract Cauchy problem on the space $\mathrm{L}^\infty\left(\left[0,1\right],\ell^1\right)$ associated to the operator $(A,\dom(A))$ defined by
\begin{align}\label{eqn:OpNetBiCont}
A=\mathrm{diag}\left(c_j\cdot\frac{\dd}{\dd{x}}\right),\quad \dom(A)=\left\{f\in\mathrm{W}^{1,\infty}\left(\left[0,1\right],\ell^1\right):\ f(1)=\mathbb{B}^Cf(0)\right\}.
\end{align}
The obvious goal is to show that the operator $(A,\dom(A))$ generates a bi-continuous semigroup on the space $X:=\mathrm{L}^\infty\left(\left[0,1\right],\ell^1\right)$. The space $X$ becomes a Banach space if equipped by the essential sup-norm
\[
\left\|f\right\|_X:=\esssup_{s\in\left[0,1\right]}{\left\|f(s)\right\|_{\ell^1}}.
\] 
Moreover, $X$ can be endowed with the weak$^*$-topology $\tau_{\mathrm{w}^*}$, as $X$ can be seen as the dual space of $Y:=\Ell^1\left(\left[0,1\right],\mathrm{c}_\mathrm{0}\right)$, and a dual space equipped with the weak$^*$-topology always satisfies Assumption \ref{asp:bicontspace}, see \cite[Ex.~1.3]{F2011Cz}. By \cite[Lemma~3.16]{BK2019} the operator $(A,\dom(A))$ is closed since the resolvent is non-empty and $\lambda-A$ is surjective for all $\lambda>0$. Moreover, by following the first part of the proof of \cite[Thm.~3.17]{BK2019} one can see that this operator is also bi-densely defined. In order to conclude that $(A,\dom(A))$ generates a bi-continuous semigroup we want to apply Theorem \ref{thm:LP1}. For this purpose it suffices to show that $(A,\dom(A))$ is bi-dissipative. 

\begin{lemma}
The operator $(A,\dom(A))$ defined by \eqref{eqn:OpNetBiCont} is bi-dissipative.
\end{lemma}

\begin{proof}
We use Proposition \ref{prop:DissSubDiff}. From \cite[Lemma~23.2]{MV1992} we obtain that $(X,\tau_{\mathrm{w}^*})'=(Y',\sigma(Y',Y))'=Y=\Ell^1\left(\left[0,1\right],\mathrm{c}_\mathrm{0}\right)$.

\medskip
Now let $\Gamma$ be a fundamental system for $\tau_{\mathrm{w}^*}$ and $p\in\Gamma$ and $f\in\dom(A)$. We define $\chi:=\left(\chi_{\left\{f\neq0\right\}}\right)=\left(\chi_{\left\{f_n\neq0\right\}}\right)_{n\in\NN}$, where $\chi_{\left\{f\neq0\right\}}$ denotes the characteristic function of the set $\left\{f\neq0\right\}$, and observe that actually $\chi\in J(f,p)$, so we only have to show that $\mathrm{Re}\left\langle Af,\chi\right\rangle\leq0$. Making use of \eqref{eqn:OpNetBiCont}, we obtain 
\begin{align*}
\mathrm{Re}\left\langle Af,\chi\right\rangle&\leq\left|\int_0^1{\left\langle (Af)(x),\chi_{\left\{f\neq0\right\}}(x)\right\rangle\ \dd{x}}\right|\leq\int_0^1{\sum_{n\in\NN}{\left|c_nf'_n(x)\chi_{\left\{f_n\neq0\right\}}(x)\right|}\ \dd{x}}\\
&=\sum_{n\in\NN}{\int_0^1{\left|c_nf'_n(x)\chi_{\left\{f\neq0\right\}}(x)\right|}\ \dd{x}}=\left\langle Cf(1)-Cf(0),\textbf{1}\right\rangle\\
&\leq\left\langle \mathbb{B}_Cf(0)-f(0),C\textbf{1}\right\rangle=\left\langle f(0),(\mathbb{B}_C^*-I)C\textbf{1}\right\rangle=0,
\end{align*} 
which concludes the proof.
\end{proof}

\section*{Acknowledgement}
The first author was funded by the DAAD-TKA Project 308019 ``\emph{Coupled systems and innovative time integrators}'' and want to acknowledges funding by the Deutsche Forschungsgemeinschaft (DFG, German Research Foundation) - 468736785. Both authors also want to thank the anonymous referee for the detailed feedback and Jochen Gl\"uck for pointing out an error in an earlier version of Example \ref{ex:heatsemi}.

\bibliographystyle{abbrv}
\bibliography{LPBiCont}

\begin{thebibliography}{10}

\bibitem{ABR2013}
A.~A. Albanese, J.~Bonet, and W.~J. Ricker.
\newblock Montel resolvents and uniformly mean ergodic semigroups of linear
  operators.
\newblock {\em Quaest. Math.}, 36(2):253--290, 2013.

\bibitem{AJ2016}
A.~A. Albanese and D.~Jornet.
\newblock Dissipative operators and additive perturbations in locally convex
  spaces.
\newblock {\em Math. Nachr.}, 289(8-9):920--949, 2016.

\bibitem{AM2004}
A.~A. Albanese and E.~Mangino.
\newblock Trotter-{K}ato theorems for bi-continuous semigroups and applications
  to {F}eller semigroups.
\newblock {\em J. Math. Anal. Appl.}, 289(2):477--492, 2004.

\bibitem{B1974}
V.~A. Babalola.
\newblock Semigroups of operators on locally convex spaces.
\newblock {\em Trans. Amer. Math. Soc.}, 199:163--179, 1974.

\bibitem{BPos}
C.~Budde.
\newblock Positive {M}iyadera-{V}oigt perturbations of bi-continuous
  semigroups.
\newblock {\em Positivity}, 25(3):1107--1129, 2021.

\bibitem{BF}
C.~Budde and B.~Farkas.
\newblock Intermediate and extrapolated spaces for bi-continuous operator
  semigroups.
\newblock {\em J. Evol. Equ.}, 19(2):321--359, 2019.

\bibitem{BF2}
C.~Budde and B.~Farkas.
\newblock A {D}esch-{S}chappacher perturbation theorem for bi-continuous
  semigroups.
\newblock {\em Math. Nachr.}, 293(6):1053--1073, 2020.

\bibitem{BK2019}
C.~Budde and M.~K. Fijav\v{z}.
\newblock Bi-continuous semigroups for flows on infinite networks.
\newblock {\em Netw. Heterog. Media}, 16(4):553--567, 2021.

\bibitem{C1987}
J.~B. Cooper.
\newblock {\em Saks spaces and applications to functional analysis}, volume 139
  of {\em North-Holland Mathematics Studies}.
\newblock North-Holland Publishing Co., Amsterdam, 1987.

\bibitem{D2008}
B.~Dorn.
\newblock Semigroups for flows in infinite networks.
\newblock {\em Semigroup Forum}, 76(2):341--356, 2008.

\bibitem{DKNR2010}
B.~Dorn, M.~K. Fijav\v{z}, R.~Nagel, and A.~Radl.
\newblock The semigroup approach to transport processes in networks.
\newblock {\em Phys. D}, 239(15):1416--1421, 2010.

\bibitem{EN}
K.-J. Engel and R.~Nagel.
\newblock {\em One-parameter semigroups for linear evolution equations}, volume
  194 of {\em Graduate Texts in Mathematics}.
\newblock Springer-Verlag, New York, 2000.

\bibitem{ESF2006}
A.~Es-Sarhir and B.~Farkas.
\newblock Perturbation for a class of transition semigroups on the {H}\"{o}lder
  space {$C^\theta_{b,{\rm loc}}(H)$}.
\newblock {\em J. Math. Anal. Appl.}, 315(2):666--685, 2006.

\bibitem{FaPhD}
B.~Farkas.
\newblock {\em Perturbations of Bi-Continuous Semigroups}.
\newblock PhD thesis, E\"otv\"os Lor\'and University, 2003.

\bibitem{F2004}
B.~Farkas.
\newblock Perturbations of bi-continuous semigroups.
\newblock {\em Studia Math.}, 161(2):147--161, 2004.

\bibitem{F2004Semi}
B.~Farkas.
\newblock Perturbations of bi-continuous semigroups with applications to
  transition semigroups on {$C_b(H)$}.
\newblock {\em Semigroup Forum}, 68(1):87--107, 2004.

\bibitem{F2011Cz}
B.~Farkas.
\newblock Adjoint bi-continuous semigroups and semigroups on the space of
  measures.
\newblock {\em Czechoslovak Math. J.}, 61(136)(2):309--322, 2011.

\bibitem{FJKW2014}
L.~Frerick, E.~Jord\'{a}, T.~Kalmes, and J.~Wengenroth.
\newblock Strongly continuous semigroups on some {F}r\'{e}chet spaces.
\newblock {\em J. Math. Anal. Appl.}, 412(1):121--124, 2014.

\bibitem{G2017}
J.~A. Goldstein.
\newblock {\em Semigroups of linear operators \& applications}.
\newblock Dover Publications, Inc., Mineola, NY, 2017.

\bibitem{GK2001}
B.~Goldys and M.~Kocan.
\newblock Diffusion semigroups in spaces of continuous functions with mixed
  topology.
\newblock {\em J. Differential Equations}, 173(1):17--39, 2001.

\bibitem{GW2016}
A.~Goli\'{n}ska and S.-A. Wegner.
\newblock Non-power bounded generators of strongly continuous semigroups.
\newblock {\em J. Math. Anal. Appl.}, 436(1):429--438, 2016.

\bibitem{JWW2015}
B.~Jacob, S.-A. Wegner, and J.~Wintermayr.
\newblock Desch-{S}chappacher perturbation of one-parameter semigroups on
  locally convex spaces.
\newblock {\em Math. Nachr.}, 288(8-9):925--934, 2015.

\bibitem{K1964}
H.~Komatsu.
\newblock Semi-groups of operators in locally convex spaces.
\newblock {\em J. Math. Soc. Japan}, 16:230--262, 1964.

\bibitem{K1968}
T.~K\={o}mura.
\newblock Semigroups of operators in locally convex spaces.
\newblock {\em J. Functional Analysis}, 2:258--296, 1968.

\bibitem{KS2005}
M.~Kramar and E.~Sikolya.
\newblock Spectral properties and asymptotic periodicity of flows in networks.
\newblock {\em Math. Z.}, 249(1):139--162, 2005.

\bibitem{Ku}
F.~K\"uhnemund.
\newblock A {H}ille--{Y}osida theorem for bi-continuous semigroups.
\newblock {\em Semigroup Forum}, 67(2):205--225, 2003.

\bibitem{LB2007}
L.~Lorenzi and M.~Bertoldi.
\newblock {\em Analytical methods for {M}arkov semigroups}, volume 283 of {\em
  Pure and Applied Mathematics (Boca Raton)}.
\newblock Chapman \& Hall/CRC, Boca Raton, FL, 2007.

\bibitem{MV1992}
R.~Meise and D.~Vogt.
\newblock {\em Einf\"{u}hrung in die {F}unktionalanalysis}, volume~62 of {\em
  Vieweg Studium: Aufbaukurs Mathematik}.
\newblock Friedr. Vieweg \& Sohn, Braunschweig, 1992.

\bibitem{Miyadera}
I.~Miyadera.
\newblock Semi-groups of operators in {F}r\'echet space and applications to
  partial differential equations.
\newblock {\em T\^ohoku Math. J. (2)}, 11:162--183, 1959.

\bibitem{O1973}
S.~\={O}uchi.
\newblock Semi-groups of operators in locally convex spaces.
\newblock {\em J. Math. Soc. Japan}, 25:265--276, 1973.

\bibitem{P1983}
A.~Pazy.
\newblock {\em Semigroups of linear operators and applications to partial
  differential equations}, volume~44 of {\em Applied Mathematical Sciences}.
\newblock Springer-Verlag, New York, 1983.

\bibitem{W2014}
S.-A. Wegner.
\newblock Universal extrapolation spaces for $\text{{C}}_0$-semigroups.
\newblock {\em Ann. Univ. Ferrara}, 60(2):447--463, 2014.

\bibitem{W2016}
S.-A. Wegner.
\newblock The growth bound for strongly continuous semigroups on {F}r\'echet
  spaces.
\newblock {\em Proc. Edinb. Math. Soc.}, 59(3):801--810, 2016.

\bibitem{W1961}
A.~Wiweger.
\newblock Linear spaces with mixed topology.
\newblock {\em Studia Math.}, 20:47--68, 1961.

\bibitem{Yosida}
K.~Yosida.
\newblock {\em Functional analysis}.
\newblock Classics in Mathematics. Springer-Verlag, Berlin, 1995.
\newblock Reprint of the sixth (1980) edition.

\end{thebibliography}

\end{document}